\documentclass[preprint]{elsarticle}

\usepackage{graphicx,subfigure}
\usepackage{amsmath}
\usepackage{amsfonts}
\usepackage{amsthm}
\usepackage{latexsym,amssymb}

\newtheorem{thm}{Theorem}

\newtheorem{lem}{Lemma}

\newdefinition{rmk}{Remark}
\newdefinition{definition}{Definition}
\newdefinition{obs}{Observation}

\begin{document}

\title{On independent $[1,2]$-sets in trees}

\author[ual]{S.A. Aleid\fnref{fn1}}
\ead{sahar.aleid@ual.es}
\author[ual]{J. C\'aceres\fnref{fn2}}
\ead{jcaceres@ual.es}
\author[ual]{M.L. Puertas\corref{cor1}\fnref{fn2}}
\ead{mpuertas@ual.es}

\cortext[cor1]{Corresponding author}

\fntext[fn1]{Supported by Phoneix Proyect (Erasmus Mundus Programme).}
\fntext[fn2]{Partially supported by Junta de Andaluc\'ia FQM305 and MTM2014-60127-P.}

\address[ual]{Department of Mathematics, University of Almer\'ia.}

\begin{abstract}
An independent $[1,k]$-set $S$ in a graph $G$ is a dominating set which is independent and such that every vertex not in $S$ has at most $k$ neighbors in it. The existence of such sets is not guaranteed in every graph and trees having an independent $[1,k]$-set have been characterized. In this paper we solve some problems previously posed by other authors about independent $[1,2]$-sets. We provide a necessary condition for a graph to have an independent $[1,2]$-set, in terms of spanning trees and we prove that this condition is also sufficient for cactus graphs. We follow the concept of excellent tree and characterize the family of trees such that any vertex belong to some independent $[1,2]$-set. Finally we describe a linear algorithm to decide whether a tree has an independent $[1,2]$-set. Such algorithm can be easily modified to obtain the cardinality of the smallest independent $[1,2]$-set of a tree.
\end{abstract}


\maketitle

\section{Introduction}

All the graphs considered here are finite, undirected, simple and connected. Undefined basic concepts can be found in introductory graph theoretical
li\-te\-ra\-tu\-re as~\cite{chartrand,haynes}. Let $G=(V,E)$ be a graph, the \emph{open neighborhood} of a vertex $v\in V$ is the set $N(v)=\{u|uv\in E\}$ of vertices adjacent to $v$. Each vertex $u\in N(v)$ is called a \emph{neighbor} of $v$. The \emph{closed neighborhood} of a vertex $v\in V$ is the set $N[v]=N(v)\cup\{v\}$. The open neighborhood of a set $S\subseteq V$ of vertices is $N(S)=\cup_{v\in S} N(v)$, while the closed neighborhood of a set $S$ is $N[S]=\cup_{v\in S} N[v]$. A set $S$ is \emph{independent} if no two vertices in $S$ are adjacent. A set $S$ is a \emph{dominating set} of a graph $G$ if $N[S]=V$, that is, for every $v\in V$, either $v\in S$ or $v\in N(u)$ for some vertex $u\in S$. A dominating set that is independent is an \emph{independent dominating} set.

In~\cite{chellali1}, Chellali et al. define a subset $S\subseteq V$ in a graph $G$ to be a \emph{[j,k]-set} if for every vertex $v\in V\setminus S$, $j\leq |N(v)\cap S|\leq k$, that is every vertex in $ V\setminus S$ is adjacent to at least $j$ vertices, but not more than $k$ vertices in $S$. In~\cite{chellali2} a similar definition was introduced with the additional condition of independence, and the minimum cardinality of an independent $[j,k]$-set is denoted by \emph{$i_{[j,k]}(G)$}. Note that the existence of such sets is not guaranteed in every graph and a characterization of trees having an independent $[1,k]$-set can be found in~\cite{chellali2}.

In this paper we focus on independent $[1,2]$-sets, that is an independent dominating set $S$ of a graph $G$ such that every vertex $u\in V(G)\setminus S$ has at most two neighbors in $S$. A number of open problems about this type of domination sets are posed in~\cite{chellali2}. In Section~\ref{sec:spanning} we give a necessary condition for a graph $G$ to have an independent $[1,2]$-set in terms of its spanning trees, that is an answer to Problem 2. This necessary condition becomes also sufficient in the class of cactus graphs, that gives a partial answer to Problem 1.

We also study the trees having an independent $[1,2]$-set from a different point of view. In Section~\ref{sec:excellent} we follow the concept of excellent tree proposed in~\cite{fricke} and we adapt it to the environment of our study, providing a characterization of trees such that any vertex belong to an independent $[1,2]$-set, that is not necessarily minimum.

The characterization of trees having an independent $[1,2]$-set of~\cite{chellali2} does not allow to obtain a polynomial algorithm solving this decision problem, so we devoted Section~\ref{sec:algorithm} to describe a linear algorithm to decide whether a tree has an independent $[1,2]$-set. This algorithm can be easily modified to obtain the cardinality of the smallest independent $[1,2]$-set of a tree, therefore we can compute $i_{[1,2]}(T)$ that solves the part of Problem 8 of~\cite{chellali2} regarding this parameter.

\section{Spanning trees}\label{sec:spanning}

In this section we provide a necessary condition for a graph $G$ to have an independent $[1,2]$-set, in terms of its spanning trees, which gives an answer to Problem 2 posed in~\cite{chellali2}. Recall that a \emph{spanning tree} of a graph $G$ is a subgraph that includes all the vertices of $G$ and that is a tree. In addition we show that this condition is also sufficient in the family of cactus graphs, which gives a partial answer to Problem 1.

To this end we will need the family ${\mathcal F}_2$ of trees having an independent $[1,2]$-set given in Theorem 11 of \cite{chellali2}. For the sake of completeness we sketch here the construction. As a first step the family of $p_2$-trees is defined in the following way. Let $T$ be a non-trivial tree and let $V(T)=X\cup Y$ be the unique bipartition of the vertex set. A tree $T$ is called a \emph{$p_2$-tree} if every vertex in one of the partite sets has degree at most $2$ and such a partite set is called a {\it $p_2$-set}. It is clear that if $X$ is a $p_2$-set of $T$ then $Y$ is an independent $[1,2]$-set of $T$. Finally Theorem 11 of~\cite{chellali2} states that a non-trivial tree $T$ admits an independent $[1,2]$-set if and only if $T$ can be obtained from a family $T_1,\dots, T_t$ of $p_2$-trees adding $t-1$ edges where each edge joins vertices in two different sets $X_i$ and $X_j$.

We call the family of trees $\mathcal{G}(T)=\{T_1,\dots, T_t\}$ a \emph{generating family} of $T$ and therefore trees in family ${\mathcal F}_2$ are those trees having a generating family. We would like to point out that the proof of Theorem 11 of~\cite{chellali2} also shows the correspondence between generating families and independent $[1,2]$-sets in a tree $T$. We recall this relationship in the following definition.

\begin{definition}
Let $T\in {\mathcal F}_2$. The independent $[1,2]$-set associated to the generating family $\mathcal{G}(T)=\{T_1,\dots, T_t\}$ is $S=\bigcup_{i=1}^t Y_i$, where $V(T_i)=X_i\cup Y_i$ is the bipartition into a $p_2$-set $X_i$ and an independent $[1,2]$-set $Y_i$.

Conversely the generating family associated to an independent $[1,2]$-set $S$ is the family of trees of the forest resulting of removing from $T$ all edges with both vertices in $V(T)\setminus S$.
\end{definition}

The necessary condition for a graph $G$ to have an independent $[1,2]$-set is shown in the following result.

\begin{thm}\label{thm:spanning trees}
Let $G$ be a graph having an independent $[1,2]$-set. Then there exists a spanning tree $T$ of $G$ satisfying $T\in \mathcal{F}_2$ and having a generating family  $\mathcal{G}(T)=\{ T_1,\dots, T_t\}$ with $V(T_j)=X_j\cup Y_j$ the bipartition into a $p_2$-set and an independent $[1,2]$-set respectively, such that any edge $e=uv\in E(G)\setminus E(T)$ satisfies either $u,v\in \bigcup_{j=1}^{t} X_j$ (type A edge) or there exists $j_e\in \{1,\dots, t\}$ such that $u$ is a leaf of $T_{j_e}$, $u\in X_{j_e}$ and $v\in Y_{j_e}$ (type B edge).
\end{thm}

\begin{proof}
Let $G$ be a graph having an independent $[1,2]$-set $S$. If $G$ is a tree the conditions are trivially true. Now suppose that $G$ has an induced cycle $C_0$. If there exists an edge $e$ in $C_0$ with both vertices in $V(G)\setminus S$ then pick $e_1=e$ (we call this case A), if each edge of $C_0$ has exactly one vertex in $S$, take $e_1$ any edge of $C_0$ (we call this case B). We define $G_1=G-e_1$, if it is not a tree it has an induced cycle $C_1$. Again either there is an edge $e_2=u_2v_2$ in $C_1$ such that $u_2,v_2\notin S$ (case A) or every edge of $C_1$ has exactly one vertex in $S$ (case B). For the second case although $C_0$ and $C_1$ could share same edges, we can take $e_2$ an edge of $C_1$ which is not an edge of $C_0$, because $C_1$ is an induced cycle in $G_1$ however vertices of $C_0$ do not induce a cycle in $G_1$. We repeat this process until we obtain $G_k=G-\{e_1,\dots ,e_k\}$ a spanning tree of $G$, where each edge $e_i=u_iv_i$ belong to $C_{i-1}$ an induced cycle of $G_{i-1}=G-\{e_1,\dots e_{i-1}\}$ and satisfies either $u_i,v_i\notin S$ or $e_i$ is not an edge of any cycle $C_{r}, r<i-1$ and every edge in $C_{i-1}$ has exactly one vertex in $S$.

Now note that $S$ is also an independent $[1,2]$-set of $G_k$, because removing edges from $G$ does not affect independence and both cases A and B ensure that $S$ dominates $G_k$. So $G_k\in \mathcal{F}_2$ and we can take $\mathcal{G}(G_k)=\{ T_1,\dots, T_t\}$ the generating family of $G_k$ associated to $S$. If edge $e_i=u_iv_i$ is in case A then $u_i,v_i\in V(G_k)\setminus S=\bigcup_{j=1}^{t} X_j$ (type A edge). If $e_i=u_iv_i$ is in case B, then every edge of $C_{i-1}$ has exactly one vertex in $S$ and note that no other edge of $C_{i-1}$ will be removed in successive steps of the construction of $G_k$, so $u_i,v_i$ are connected in $G_k$ by the path $C_{i-1}-e_i$, where each edge has one vertex in $S$. This means that $u_i,v_i$ are in the same connected component of the forest resulting of removing from $G_k$ all edges with no vertex in $S$, or equivalently that there exists $j_i\in \{1,\dots, t\}$ with $u_i,v_i\in V(T_{j_i})$. Moreover $u_i\notin S$ and $v_i\in S$ gives $u_i\in X_{j_i}$ and $v_i\in Y_{j_i}$. Finally both neighbors of $u_i$ in the cycle $C_{i-1}$ belong to $S$, so if $u_i$ has any other neighbor $z$ in $G$, which is not in $C_{i-1}$ it is clear that $z\notin S$ so edge $u_iz$ joints two different trees of the generating family $\mathcal{G}(G_k)=\{ T_1,\dots, T_t\}$. That means $u_i$ is a leaf of $T_{j_i}$ (type B edge).
\end{proof}

The following example shows that the converse of Theorem~\ref{thm:spanning trees} is not true in general. The graph in Figure~\ref{fig:no_spanning} has no independent $[1,2]$-set because all black vertices should be in such set, so vertex $v$ would have three neighbors in that set. However the set of black vertices is an independent $[1,2]$-set of the tree in Figure~\ref{fig:spanning}, which is the spanning tree of $G$ resulting from removing edges $e_1$ and $e_2$, which are type B.

\par\medskip

\begin{figure}[htbp]
  \begin{centering}
    \subfigure[$G$ has no independent ${[}1,2{]}$-set]{\includegraphics[width=0.45\textwidth]{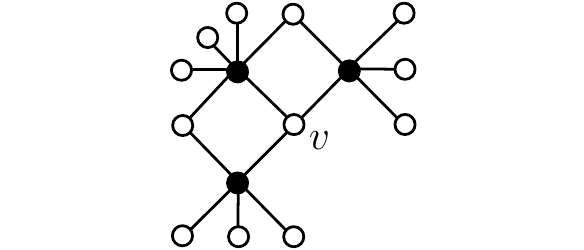}\label{fig:no_spanning}}
    \hspace{1cm}
    \subfigure[$T$ is a $p_2$ tree that spans $G$]{\includegraphics[width=0.45\textwidth]{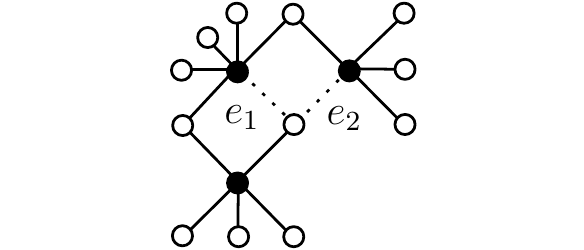}\label{fig:spanning}}
    \caption{The converse of Theorem~\ref{thm:spanning trees} is not true in general}\label{fig:counterexample}
  \end{centering}
\end{figure}

The key point of this counterexample is that the spanning tree is obtained from $G$ removing some edges such that at least one of them belong to two induced cycles, in that example the edge $e_1$. This idea leads us to the family of cactus graphs where the necessary condition to have an independent $[1,2]$-set showed in Theorem~\ref{thm:spanning trees} is also sufficient. Recall that $G$ is a \emph{cactus graph} if every edge of $G$ belongs to at most one cycle. Equivalently $G$ is a cactus graph if and only if every block (maximal connected induced subgraph with no cut vertices) is a cycle or the path $P_2$.

\begin{thm}\label{thm:cactus graphs}
Let $G$ be a cactus graph. Then $G$ has an independent $[1,2]$-set if and only if there exists a spanning tree $T$ of $G$ satisfying $T\in \mathcal{F}_2$ and having a generating family  $\mathcal{G}(T)=\{ T_1,\dots, T_t\}$ with $V(T_j)=X_j\cup Y_j$ the bipartition into a $p_2$-set and an independent $[1,2]$-set respectively, such that any edge $e=uv\in E(G)\setminus E(T)$ satisfies either $u,v\in \bigcup_{j=1}^{t} X_j$ (type A edge) or there exists $j_e\in \{1,\dots, t\}$ with $u$ a leaf of $T_{j_e}$, $u\in X_{j_e}$ and $v\in Y_{j_e}$ (type B edge).
\end{thm}

\begin{proof}
We just need to prove the sufficiency. Let $S=\bigcup_{j=1}^t Y_j$ be the independent $[1,2]$-set of $T$ associated to the generating family $\mathcal{G}(T)$ and let us see that it is also an independent $[1,2]$-set of $G$. The graph $G$ is obtained from the spanning tree $T$ adding some edges, so $S$ is also a dominating set of $G$. Moreover, by hypothesis no added edge has both vertices in $S$, therefore $S$ is independent in $G$. Finally we need to show that $S$ is a $[1,2]$-set of $G$. Let $x\in V(G)\setminus S$, if every edge of $G$ incident to $x$ is an edge of $T$, then $N_G(x)=N_T(x)$ and $x$ has at most two neighbors in $S$. On the contrary suppose that the set of edges incident with $x$ which are in $E(G)\setminus E(T)$ is non-empty and denote those edges as $e_1,\dots e_r$ with $e_i=xy_i$. Using that $G$ is a cactus graph and that removing theses edges does not disconnect the graph, each edge $e_i$ belong to exactly one cycle $C_i$ in $G$, with $C_i\neq C_j$ for $i\neq j$, and $x$ is a common vertex of all of them (see Figure~\ref{fig:cactus1}).

Firstly suppose that all edges $e_1,\dots e_r$ are of type A, that is $y_i\in \bigcup_{j=1}^{t} X_j =V(G)\setminus S, \forall i=1,\dots r$. Then the neighbors of $x$ in $G$ other than $y_1,\dots, y_r$, are also neighbors on $x$ in $T$ so it is clear that $x$ has at most two neighbors in $G$ belonging to $S$. On the other hand suppose, without loss of generality, that $e_1=xy_1$ is type B, so there exists $j\in \{1,\dots,t\}$ such that $x$ is a leaf of $T_{j}$, $x\in X_{j}$ and $y_1\in Y_{j}$. Therefore $x$ has just one neighbor in $T_{j}$, say $z_1$, which is in $Y_{j}$, and both $y_1,z_1$ are neighbors of $x$ in $G$ belonging to $S$.

Let $w\in N_G(x)\setminus\{y_1,z_1\}$, if $w\in N_T(x)$ then $w$ belongs to a tree $T_l\neq T_{j}$, the edge $xw$ connects two different trees of the forest $T_1,\dots, T_t$, and by construction $w\notin S$. Finally if $w\notin N_T(x)$, then $w\in \{y_2,\dots, y_k\}$, say $w=y_2$. Vertex $y_2$ belongs to cycle $C_2$ in $G$, different from cycle $C_1$ containing $y_1$, and we denote the neighbor of $x$ in $C_2$, other than $y_2$, by $z_2$. Using that $x$ is a leaf of $T_j$ with neighbor $z_1$, which is a vertex of cycle $C_1$, we obtain that $z_2\neq z_1$, the edge $xz_2$ does not belong to $T_j$ and thus $z_2\notin V(T_j)$. So $z_2$ belongs to a tree of the forest $T_1,\dots, T_t$ different form $T_j$ and $y_2$ belongs to the same one. Therefore $y_2$ does not belong to $V(T_j)$ (see Figure~\ref{fig:cactus2}). This means that edge $xy_2$ must be of type A and $w=y_2\notin S$ as desired.
\end{proof}

\begin{figure}[htbp]
  \begin{centering}
    \subfigure[Cycles sharing vertex $x$ in the cactus graph $G$]{\includegraphics[width=0.45\textwidth]{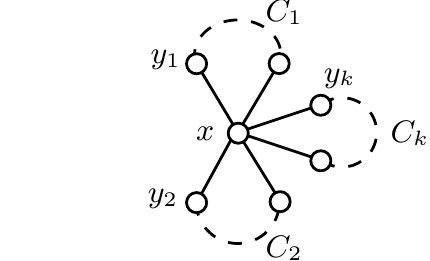}\label{fig:cactus1}}
    \hspace{1cm}
    \subfigure[Edge $x z_2$ connects two trees of forest $T_1,\dots, T_t$ and $x$ is a leaf of $T_j$]{\includegraphics[width=0.45\textwidth]{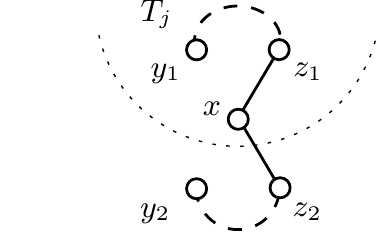}\label{fig:cactus2}}
    \caption{Some cases of Theorem~\ref{thm:cactus graphs}}\label{fig:cactus}
  \end{centering}
\end{figure}

\section{Excellent trees}\label{sec:excellent}

We recall now the concept of \emph{excellent graph} introduced in \cite{fricke}. For a graph $G=(V,E)$, let $\mathcal{P}$ denote a property of subsets $S\subseteq V$. We call a set $S$ with property $\mathcal{P}$ having $\{$minimum, maximum$\}$ cardinality $\mu(G)$ a $\mu(G)$-set. A vertex is called \emph{$\mu$-good} if it is contained in some $\mu(G)$-set. A graph G is called {\it $\mu$-excellent} if every vertex in $V$ is $\mu$-good. For instance $G$ is $\gamma$-excellent if every vertex of $G$ belong to a minimum dominating set. This concept has been studied in the family of trees for different domination-type properties such as domination, irredundance and independence~\cite{fricke,hahe}, restrained domination~\cite{hattingh}  and total domination~\cite{henning}.

We define a similar concept for the independent $[1,2]$-domination and having in mind that the existence of such sets is a key problem so we relax the conditions in the following way.

\begin{definition}
A graph $G$ is \emph{[1,2]-semiexcellent} if every vertex belongs to some independent $[1,2]$-set, not necessarily minimum.
\end{definition}

Our target is to characterize the family of trees that are $[1,2]$-semiexcellent and to this end we will again use the concept of $p_2$-tree and the family ${\mathcal F}_2$ described in Section~\ref{sec:spanning}. Firstly we show a necessary condition for a vertex in order to belong to some independent $[1,2]$-set.

\begin{lem}\label{lem:oneleaf}
Let $T$ be a tree and let $v\in V(T)$. Suppose that there exists an independent $[1,2]$-set $S_v$ containing $v$, then for each $u\in N(v)$, the set $N(u)\setminus\{v\}$ contains at most one leaf.
\end{lem}

\begin{proof}
If $S_v$ is an independent $[1,2]$-set containing $v$ and  $u\in N(v)$, it is clear that $u\notin S_v$ and any leaf in $N(u)\setminus \{v\}$ must belong to $S_v$ in order to be dominated, so $N(u)\setminus \{v\}$ can have at most one leaf because $u$ has at most two neighbors in $S_v$.
\end{proof}

The following lemma shows that this condition is also sufficient in the family of $p_2$-trees.

\begin{lem}\label{lem:p2-trees}

Let $T$ be a $p_2$-tree with $V(T)=X\cup Y$ the bipartition into a $p_2$-set $X$ and an independent $[1,2]$-set $Y$. Let $x\in X$ be such that any $y\in N(x)$ satisfies that $N(y)\setminus \{ x\}$ contains at most one leaf. Then $S_x=\big (Y\setminus N(x)\big)\cup L(x) \cup \{x\}$ is an independent $[1,2]$-set of $T$ containing $x$, where $L(x)$ is the set of leaves at distance two of $x$.

\end{lem}

\begin{proof}

Firstly the set $S_x=\big (Y\setminus N(x)\big)\cup L(x) \cup \{x\}$ is independent because $Y$ is independent and all neighbors of vertices in $L(x)\cup \{x\}$ belong to $N(x)$. Let us see that $S_x$ is a $[1,2]$-set. Let $y\in N(x)$, it is clear that $y$ is dominated by $x$ and using the hypothesis that $N(y)\setminus \{x\}$ has at most one leaf, there is at most one vertex in $L(x)$ that dominates $y$.

On the other hand if $z\in V(T)\setminus S_x$ and $z\in N(y)$ for some $y\in N(x)$, then it is not a leaf so it has degree 2, because $X$ is $p_2$-set. Therefore $z$ has a unique neighbor $y'\neq y$ and it satisfies $y'\in S_x$ (see Figure~\ref{ex:p2tress}).

Finally let $t\in V(T)\setminus S_x$ be such that $t\notin N[y]$ for any $y\in N(x)$. Then $t\in X$ has no neighbors in $L(x)\cup \{x\}$ and it has at least one and at most two neighbors in $Y\setminus N(x)$.
\end{proof}

\begin{figure}[htbp]
\begin{center}
\includegraphics[width=0.35\textwidth]{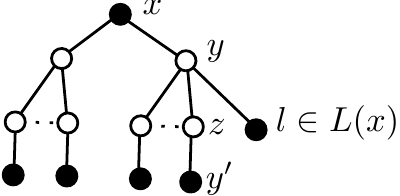}
\caption{Black vertices are in $S_x$ and white vertices are not in $S_x$.}
\label{ex:p2tress}
\end{center}
\end{figure}

Recall that a \emph{strong support} vertex is a vertex having at least two leaves in its neighborhood. In the family of $p_2$-trees it is also possible to obtain an independent $[1,2]$-set that skips a fixed pair of adjacent vertices, under the condition of having no strong support vertices.

\begin{lem}\label{lem:pair}
Let $T$ be a $p_2$-tree with no strong support vertices, $V(T)=X\cup Y$ the bipartition into a $p_2$-set $X$ and an independent $[1,2]$-set $Y$, and let $x,y\in V(T)$ be two adjacent vertices such that none of them is a leaf, $x\in X$ and $y\in Y$. Then there exists an independent $[1,2]$-set $S(x,y)$ such that $x,y\notin S(x,y)$ and $x$ has just one neighbor in $S(x,y)$.
\end{lem}

\begin{proof}
Let $T$ be a $p_2$-tree with no strong support vertices, $V(T)=X\cup Y$ the bipartition into a $p_2$-set $X$ and an independent $[1,2]$-set $Y$ and let $x,y\in V(T)$ be two adjacent non-leaves vertices, $x\in X$ and $y\in Y$. Using that $y$ is not a leaf, the set $N(y)\setminus \{x\}$ is non-empty. Firstly suppose that $N(y)$ contains a leaf $x_1$, that is unique by hypothesis. Then $S(x,y)=\big(Y\setminus\{y\}\big) \cup \{ x_1\}$ is an independent $[1,2]$-set of $T$ with $x,y\notin S(x,y)$ and such that $x$ has just one neighbor in it (see Figure~\ref{fig:pair_a}).

On the contrary suppose that $N(y)$ contains no leaves and take any vertex $x_1\in N(y)\setminus \{x\}$. Then $x_1$ has degree $2$ and let $y_1$ be a neighbor of $x_1$ other than $y$. If $y_1$ is a leaf or if $N(y_1)$ contains no leaves, then define $S(x,y)=\big(Y\setminus \{y,y_1\}\big) \cup \{x_1\}$ (see Figure~\ref{fig:pair_b}). If $N(y_1)$ contains a (unique) leaf, say $x_2$, then define $S(x,y)=\big(Y\setminus \{y,y_1\}\big) \cup \{x_1,x_2\}$ (see Figure~\ref{fig:pair_c}). In any case $S(x,y)$ is an independent $[1,2]$-set of $T$ with $x,y\notin S(x,y)$ and such that $x$ has just one neighbor in it.
\end{proof}

\par\medskip

\begin{figure}[htbp]
  \begin{centering}
    \subfigure[]{\includegraphics[width=0.3\textwidth]{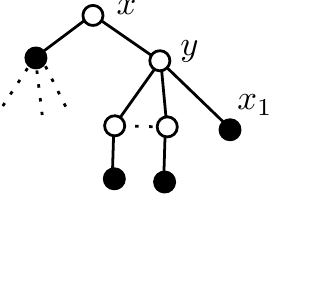}\label{fig:pair_a}} \hspace{0.25cm}
    \subfigure[]{\includegraphics[width=0.3\textwidth]{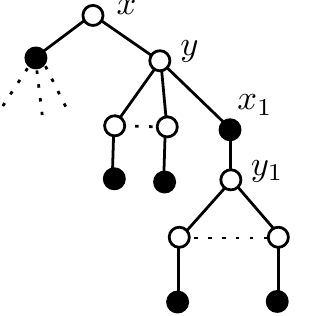}\label{fig:pair_b}}\hspace{0.25cm}
    \subfigure[]{\includegraphics[width=0.3\textwidth]{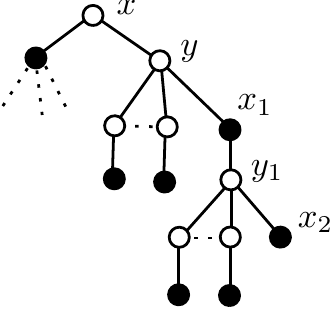}\label{fig:pair_c}}
\caption{Black vertices are in $S(x,y)$.}\label{fig:pair}
  \end{centering}
\end{figure}

The last lemma of this section shows that having no strong support vertices is a sufficient condition for a tree for belonging to the family $\mathcal{F}_2$.

\begin{lem}\label{lem:f2-tress}
Let $T$ be a tree with no strong support vertices, then $T\in \mathcal{F}_2$.
\end{lem}

\begin{proof}
We root the tree $T$ in a leaf $v$ and we label the vertices of $T$ as $X$ or $Y$ with the following rules. First of all we label $v$ as $X$ and its unique neighbor as $Y$. All the children of any vertex labeled as $Y$ are labeled as $X$. If a vertex with label $X$ has just one child we label it as $Y$. If a vertex with label $X$ has two or more children and (just) one of them is a leaf, we label this leaf as $Y$ and the rest of children as $X$ and finally if a vertex with label $X$ has two or more children and none of them is a leaf, we label one of the children as $Y$ and the rest of children as $X$.

Removing all edges of $T$ between two vertices labeled as $X$ gives a forest $T_1,T_2,\dots ,T_t$ and note that each $T_i$ is a $p_2$-tree where vertices labeled as $X$ are a $p_2$-set and vertices labeled as $Y$ are an independent $[1,2]$-set. So we obtain a generating family for $T$ and $T\in \mathcal{F}_2$ as desired.
\end{proof}

Finally we show the characterization of $[1,2]$-semiexcellent trees, as trees having no strong support vertices with the exception of the path $P_3$.

\begin{thm}
Let $T$ be a tree, $T\neq P_3$. Then $T$ is $[1,2]$-semiexcellent if and only if $T$ has no strong support vertices.
 \end{thm}

\begin{proof}
Suppose that $T$ is $[1,2]$-semiexcellent and that $v\in V(T)$ is a strong support vertex of $T$. Let $u_1,u_2 \in N(v)$ be two leaves of $T$. Using that $T\neq P_3$, there exists $w\in N(v)\setminus \{u_1,u_2\}$. By hypothesis there exists an independent $[1,2]$-set $S_w$ containing $w$ and by Lemma~\ref{lem:oneleaf} the set $N(v)\setminus\{w\}$ contains at most one leaf, that contradicts the fact $u_1,u_2 \in N(v)\setminus\{w\}$.

Conversely suppose that $T$ has no strong support vertices and let $v\in V(T)$. By Lemma~\ref{lem:f2-tress}, $T\in \mathcal{F}_2$ so let $\mathcal{G}(T)=\{ T_1,\dots, T_t\}$ be a generating family for $T$. If $v\in Y_i$ for some $i\in \{1,\dots ,t\}$ then $v\in Y=\bigcup_{j=1}^t Y_j$, that is an independent $[1,2]$-set of $T$. So suppose that $v=x\in X=\bigcup _{j=1}^t X_j$ and without loss of generality consider the case $x\in X_1$. We are going to construct an independent $[1,2]$-set of $T$ containing $x$.

By Lemma~\ref{lem:p2-trees}, the set $S_x=\big (Y_1\setminus N_{T_1}(x)\big)\cup L_{T_1}(x) \cup \{x\}$ is an independent $[1,2]$-set if $T_1$, so it is clear that $S^1=S_x \cup \big(\bigcup_{j=2}^t Y_j \big)$ is independent and dominates $T$. If $S^1$ is a $[1,2]$-set we are done. On the contrary if there exits $u\in V(T)\setminus S^1$ with more than two neighbors in $S^1$ it must be (w.l.o.g.) $x_2\in X_2$ with exactly one neighbor in $S_x$ (by definition of the generating family $\mathcal{G}(T)$) and two neighbors in $Y_2$, at least one of them, say $y_2$, is not a leaf of $T$ because $T$ has no strong support vertices. Using Lemma~\ref{lem:pair}, let $S_2 =S(x_2,y_2)$ be an independent $[1,2]$-set of $T_2$ such that $x_2,y_2\notin S(x_2,y_2)$ and $x_2$ is dominated just once. Now we call $S^2=S_x\cup S_2\cup \big(\bigcup_{j=3}^t Y_j \big)$. Again $S^2$ is an independent dominating set of $T$, if it is also a $[1,2]$-set then we are done. If it is not the case, there exists $x_3\in X_3$ (w.l.o.g.) with exactly one neighbor in $S_x\cup S_2$ (again by definition of the generating family $\mathcal{G}(T)$) and two neighbors in $Y_3$. We repeat the same construction in $T_3$ as in $T_2$ (see Figure~\ref{fig:excellent}). Iterating the process as many times as necessary we finally obtain $S^r=S_x\cup S_2\cup \dots \cup S_r \cup \big(\bigcup_{j=r+1}^t Y_j \big)$ which is an independent $[1,2]$-set of $T$ containing $x$.
\end{proof}

\begin{figure}[htbp]
\begin{center}
\includegraphics[width=0.40\textwidth]{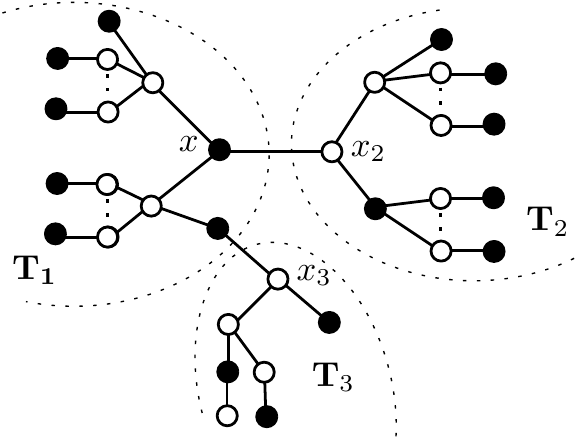}
\caption{Black vertices are in the independent $[1,2]$-set containing $x$.}
\label{fig:excellent}
\end{center}
\end{figure}

\section{A linear algorithm for trees}\label{sec:algorithm}

The characterization of trees having an independent $[1,2]$-set shown in Theorem 11 of~\cite{chellali2} does not allow to devise a polynomial algorithm to solve this decision problem. In this final section we focus on providing such algorithm for this graph class. In addition our algorithm can be easily modified to obtain the cardinality of the smallest independent $[1,2]$-set of a tree, which provides the answer to Problem 8 of~\cite{chellali2} regarding the parameter $i_{[1,2]}(T)$. We begin with the definition of the next labeling of vertices.

\begin{definition}
Let $G$ be a graph with at least two vertices, and let $v\in V(G)$. An independent vertex set $S\subseteq V(G)$ is of {\emph type I for $v$} if every vertex $u\in V(G)\setminus (S\cup \{v\})$ has at least one and at most two vertices in $S$ and $v$ is either in $S$ or it is not in $S$ and has zero, one or two neighbors in $S$. We denote $I(v,G)$ the family of type I sets for $v$ in $G$. Given $S\in I(v,G)$ we define the following labeling of $v$.

\par\bigskip

$L_S(v)=
\left\{
\begin{array}{rl}
0 & \text{if } v\in S \\
k &  \text{if } v\notin S \text{ and $v$ has $k$ neighbors in }S, k \geq 1 \\
-1& \text{if } N[v]\cap S=\emptyset \text{ and every neighbor of $v$ has exactly one }\\
& \text{neighbor in }S \\
-2& \text{if } N[v]\cap S=\emptyset \text{ and there exists a neighbor of $v$ having }\\
&\text{two neighbors in }S\\
\end{array}
\right.
$
\end{definition}

\begin{rmk}\label{rmk:labeling}
Note that any independent $[1,2]$-set of $G$ is of type I for every vertex in $G$. It is also clear from the definition that if there exists $R\in I(v,G)$ with $L_R(v)=-1$ then $S=R\cup \{v\}$ satisfies $S\in I(v,G)$ and $L_S(v)=0$.
\end{rmk}

The following lemma is straightforward.
\begin{lem}\label{lem:star}
Let $K_{1,r}$ be the star with $x_1,\dots x_r, (r\geq 1)$ leaves and center $v$. If $r\geq 3$ then the unique type I set for $v$ is $S=\{v\}$ and $L_S(v)=0$. If $r=2$ then $S=\{v\}$ and $S'=\{x_1,x_2\}$ are the unique type I sets for $v$ and they satisfy $L_S(v)=0$, $L_{S'}(v)=2$. If $r=1$ then $S=\{v\}$ and $S'=\{x_1\}$ are the unique type I sets for $v$ and they satisfy $L_S(v)=0$, $L_{S'}(v)=1$.
\end{lem}

In the following lemma we add one new vertex and just one edge to a graph and we show how to obtain all type I sets for the new vertex.

\begin{lem}\label{lem:onevertex}
Let $G$ be a graph with at least two vertices and let $v\in V(G)$. Let $G'$ be the graph obtained from $G$ and a new vertex $v'$ by adding edge $vv'$ and let $S'\subseteq V(G')$. Then
\begin{enumerate}
\item $S'\in I(v',G')$ and $L_{S'}(v')=-2$ if and only if $S'\in I(v,G)$ and $L_{S'}(v)=2$
\item $S'\in I(v',G')$ and $L_{S'}(v')=-1$ if and only if $S'\in I(v,G)$ and $L_{S'}(v)=1$
\item $S'\in I(v',G')$ and $L_{S'}(v')=0$ if and only if $S'=S\cup \{v'\}$, $S\in I(v,G)$ and $L_S(v)\in \{-2,-1,1\}$
\item $S'\in I(v',G')$ and $L_{S'}(v')=1$ if and only if $S'\in I(v,G)$ and $L_{S'}(v)=0$. In addition if there exists $R\in I(v,G)$ with $L_{R}(v)=-1$ then $S=R\cup \{v\}$ satisfies $S\in I(v',G')$ with $L_{S}(v')=1$.
\end{enumerate}
\end{lem}

\begin{proof}

First and second statements are clear by the definition of type I set.

Now suppose that $S'\in I(v',G')$ and $L_{S'}(v')=0$ then $v'\in S'$ and we define $S=S'\setminus \{v'\}$ which is a type I set for $v$ in $G$. If $v$ has two neighbors in $S'$ then $L_S(v)=1$, if $v'$ is the unique neighbor of $v$ in $S'$ and every other neighbor of $v$ in $G$ is dominated just one by vertices in $S'$ then $L_S(v)=-1$ and if $v'$ is the unique neighbor of $v$ in $S'$ and there exists a neighbor of $v$ in $G$ which is dominated twice by vertices in $S'$ then $L_S(v)=-2$. The converse is trivial using the definition of type I set.

To prove the last statement, just by definition, $S'\in I(v',G')$ and $L_{S'}(v')=1$ if and only if $S'\in I(v,G)$ and $L_{S'}(v)=0$. The additional implication comes from Remark~\ref{rmk:labeling}.
\end{proof}

\begin{rmk}\label{rmk:onevertex}
In addition to characterize sets $S'\in I(v',G')$, Lemma~\ref{lem:onevertex} also ensures that from any $S\in I(v,G)$ can be obtained at least one $S'\in I(v',G')$ and it shows the labeling $L_{S'}(v')$ in each case.
\end{rmk}

In the next lemma we join two graphs with one new edge and we show how to obtain all type I sets for one vertex of this edge.

\begin{lem}\label{lem:twographs}
Let $G,G'$ be two graphs with at least two vertices and let $v\in V(G), v'\in V(G')$. Let $G''$ be the graph obtained from $G$ and $G'$ by adding edge $vv'$. Then
\begin{enumerate}
\item $S''\in I(v',G'')$ and $L_{S''}(v')=-2$ if and only if $S''=S\cup S'$, $S\in I(v,G)$, $S'\in I(v',G')$ and $(L_{S}(v),L_{S'}(v'))\in \{ (1,-2), (2,-2), (2,-1)\}$.

\item $S''\in I(v',G'')$ and $L_{S''}(v')=-1$ if and only if  $S''=S\cup S'$, $S\in I(v,G)$, $S'\in I(v',G')$ and $(L_{S}(v),L_{S'}(v'))= (1,-1)$.

\item $S''\in I(v',G'')$ and $L_{S''}(v')=0$ if and only if $S''=S\cup S'$, $S\in I(v,G)$, $S'\in I(v',G')$ and $(L_{S}(v),L_{S'}(v'))\in \{ (-2,0),(-1,0), (1,0)\}$.\\
Furthermore suppose that there exists $R'\in I(v',G')$ with $L_{R'}(v')=-1$ then $S''=S\cup (R'\cup \{v'\})$, where $S\in I(v,G)$ and $L_{S}(v)\in \{ -2,-1,1 \}$, satisfies $S''\in I(v',G'')$ and $L_{S''}(v')=0$.

\item $S''\in I(v',G'')$ and $L_{S''}(v')=1$ if and only if $S''=S\cup S'$, $S\in I(v,G)$, $S'\in I(v',G')$ and $(L_{S}(v),L_{S'}(v'))\in \{ (0,-2), (0,-1), (1,1), (2,1)\}$. Furthermore suppose that there exists $R\in I(v,G)$ with $L_{R}(v)=-1$ then $S''=(R\cup \{v\})\cup S'$, where $S'\in I(v',G')$ and $L_{S'}(v')\in \{ -2,-1\}$, satisfies $S''\in I(v',G'')$ and $L_{S''}(v')=1$.

\item $S''\in I(v',G'')$ and $L_{S''}(v')=2$ if and only if $S''=S\cup S'$, $S\in I(v,G)$, $S'\in I(v',G')$ and $(L_{S}(v),L_{S'}(v'))\in \{ (0,1), (1,2), (2,2)\}$. Furthermore suppose that there exists $R\in I(v,G)$ with $L_{R}(v)=-1$ then the set $S''=(R\cup \{v\})\cup S'$, where $S'\in I(v',G')$ and $L_{S'}(v')=1$, satisfies $S''\in I(v',G'')$ and $L_{S''}(v')=2$.

\end{enumerate}
\end{lem}

\begin{proof}
The sufficient implication of each equivalence is trivial using the definition of type I set so we just prove the necessity. To this end let $S''\in I(v',G'')$ and denote by $S=S''\cap V(G)$ and $S'=S''\cap V(G')$. Using that each graph has at least two vertices, $S$ and $S'$ are non-empty sets and it is clear that $S\in I(v,G)$ and $S'\in I(v',G')$.

\begin{enumerate}

\item If $L_{S''}(v')=-2$ then $v'$ has no neighbors in $S''$ so $v\notin S''$ and $L_{S}(v)\neq 0$. If $v$ has just one neighbor $z$ in $S''$ then $z\in V(G)$ and $L_{S}(v)=1$ and using that $L_{S''}(v')=-2$ there exists a neighbor of $v'$ in $G'$ with two neighbors in $S''$ so $L_{S'}(v')=-2$. If $v$ has two neighbors in $S''$ then both of them belong to $V(G)$ and $L_{S}(v)=2$. Moreover the neighbors of $v'$ in $G'$ could have one or two neighbors in $S''$, so $L_{S'}(v')\in \{-2,-1\}$.

\item If $L_{S''}(v')=-1$ then by hypothesis every neighbor of $v'$ in $G''$ has just one neighbor in $S''$ and this easily implies that $L_{S}(v)=1$ and $L_{S'}(v')=-1$.

\item If $L_{S''}(v')=0$ then $v'\in S'$ so $L_{S'}(v')=0$. If $v$ has two neighbors in $S''$ then $L_{S}(v)=1$ and if $v'$ is the unique neighbor of $v$ in $S''$ then $L_{S}(v)\in \{-2,-1\}$. The additional implication comes from Remark~\ref{rmk:labeling}.

\item If $L_{S''}(v')=1$ then there are two cases. If $v\in S''$ then $L_{S}(v)=0$ and $L_{S'}(v')\in \{-2,-1\}$ and if $v\notin S''$ then $L_{S}(v)\in \{1,2\}$ and $L_{S'}(v')=1$. The additional implication comes from Remark~\ref{rmk:labeling}.

\item If $L_{S''}(v')=2$ then there are two cases. If $v\in S''$ then $L_{S}(v)=0$ and $L_{S'}(v')=1$ and if $v\notin S''$ then $L_{S}(v)\in\{1,2\}$ and $L_{S'}(v')=2$. The additional implication comes from Remark~\ref{rmk:labeling}.
\end{enumerate}\end{proof}

\begin{rmk}\label{rmk:twographs}
In addition to characterize sets $S''\in I(v',G'')$, Lemma~\ref{lem:twographs} also ensures that from sets $S\in I(v,G)$ and $S'\in I(v',G')$ can be obtained at least one $S''\in I(v',G'')$ if and only if $(L_{S}(v),L_{S'}(v'))\in \{(1,-2), (2,-2), (2,-1),(1,-1),\\(-2,0),(-1,0),
 (1,0),(-2,-1),(-1,-1), (1,-1),(0,-2), (0,-1), (1,1), (2,1),\\
 (-1,-2), (-1,-1), (0,1), (1,2), (2,2), (-1,1)\}$ and it shows $L_{S''}(v')$ in each case.
\end{rmk}

Finally we present a linear algorithm that decides whether or not a tree $T$ has an independent $[1,2]$-set. The algorithm defines an order in the set of non-leaf vertices and proceeds bottom up in the tree.

\noindent \texttt{Algorithm TREE-INDEPENDENT [1,2]-SET}\\
\noindent \texttt{Input: A tree $T$ with $n$ internal vertices.}\\
\noindent \texttt{Output: Whether or not $T$ admits an independent $[1,2]$-set.}\\
\noindent \texttt{choose a non-leaf vertex as the root;}\\
\noindent \texttt{label the rest of vertices with different labels, and in such a way}\\
\noindent \hspace*{5 mm}\texttt{that if $u$ is a descendant of $v$ then $i(u)<i(v)$;}\\
\noindent \texttt{initialize a list for each vertex as $R(u):=\{\}$;}\\
\noindent \texttt{for $i:=1$ to $n$ do}\\
\noindent \hspace*{5 mm}\texttt{let $v$ be the vertex with label $i$, i.e. $i(v)=i$;}\\
\noindent \hspace*{5 mm}\texttt{if $v$ is a support vertex then, apply Lemma~\ref{lem:star} to the star with $v$}\\
\noindent \hspace*{10 mm}\texttt{as center and its descendant leaves to actualize $R(v)$;}\\
\noindent \hspace*{10 mm}\texttt{for each non-leaf descendant $u$ of $v$}\\
\noindent \hspace*{15 mm}\texttt{apply Lemma 7 and actualize $R(v)$;}\\
\noindent \hspace*{5 mm}\texttt{if $v$ is not a support vertex then}\\
\noindent \hspace*{10 mm}\texttt{pick one of its descendant $u$ and apply Lemma 6 for}\\
\noindent \hspace*{15 mm}\texttt{actualizing $R(v)$;}\\
\noindent \hspace*{10 mm}\texttt{for the rest of its descendant $u$ of $v$}\\
\noindent \hspace*{15 mm}\texttt{apply Lemma 7 and actualize $R(v)$;}\\
\noindent \hspace*{5 mm}\texttt{if $R(v)=\emptyset$ then answer NO and end;}\\
\noindent \texttt{od;}\\
\noindent \texttt{if $R(v)\cap \{0,1,2\}= \emptyset$ for the root $v$ then answer NO}\\
\noindent \hspace*{5 mm}\texttt{otherwise answer YES;}\\
\noindent \texttt{end.}\\

\begin{thm}
Let $T$ be a tree with $n$ vertices. The algorithm decides in $O(n)$ time whether or not $T$ has an independent $[1,2]$-set.
\end{thm}
\begin{proof}
At any moment of the execution, the list $R(v)$ associated to the non-leaf vertex $v$ contains all the possible labels $L_S(v)$ where $S\in I(v,T_v)$ and $T_v$ is the subtree rooted in $v$. Note that if $S$ is an independent $[1,2]$-set of $T$ then $S\cap V(T_v)\in I(v,T_v)$ for each non-leaf vertex $v$, so if $R(v)=\emptyset$ for some $v$ then $T$ has no independent $[1,2]$-set. Moreover at the end, there is a independent $[1,2]$-set if and only if $R(v)\cap \{0,1,2\}\neq\emptyset$ where $v$ is the root.

Regarding the complexity, the initial part is done in linear time. In the rest of the algorithm, every vertex different from the root and the leaves is considered twice and the operations over it are done in constant time. Hence, the final complexity is in $O(n)$.
\end{proof}

Although we have preferred to introduce the algorithm in its present form for the sake of simplicity, it would not be difficult to modify it in order to keep track of the minimum cardinality of the possible sets associated to a label in any vertex. Then we obtain a linear algorithm for computing the parameter $i_{[1,2]}(T)$, solving part of Problem 8 posed in~\cite{chellali2}.

In the following examples we show a tree having an independent $[1,2]$-set and another one that has no such set. The vertex indexes appear inside the circles. Figure~\ref{fig:finalstep_yes} shows a tree and the final assignment of labels to every non-leaf vertex. The root $v$ satisfies $R(v)=\{2\}$, so the tree has an independent $[1,2]$-set.
On the other hand in Figure~\ref{fig:finalstep_no} we show a different tree such that the root has no suitable label at the end of the algorithm, therefore the tree has no independent $[1,2]$-set.

\begin{figure}[htbp]
  \begin{centering}
    \subfigure[Final step on the algorithm with labels in any non-leaf vertex, including the root.] {\includegraphics[width=0.40\textwidth]{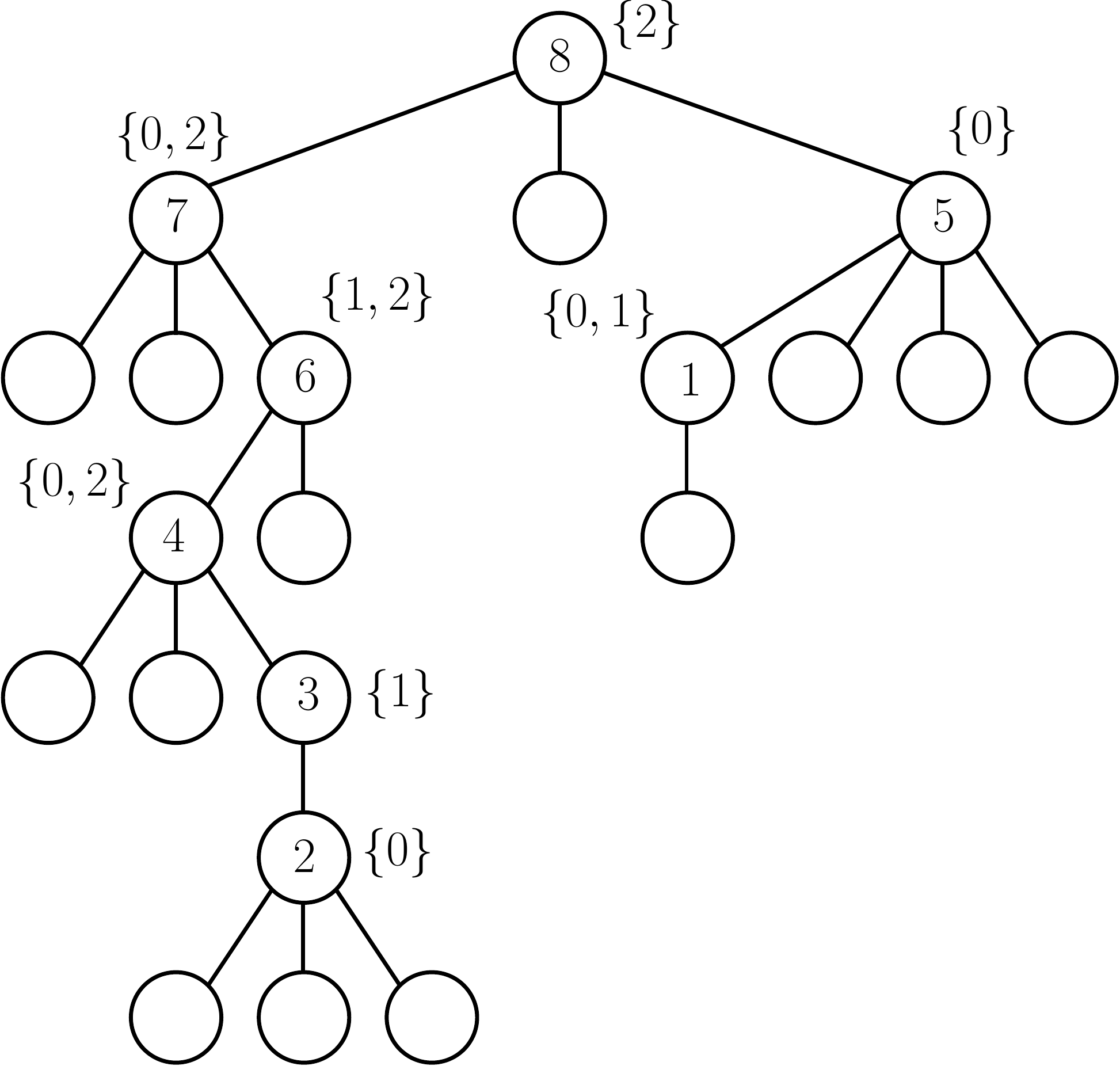}\label{fig:finalstep_yes}}
    \hspace{1cm}
    \subfigure[Final step on the algorithm with no suitable label for the root.] {\includegraphics[width=0.40\textwidth]{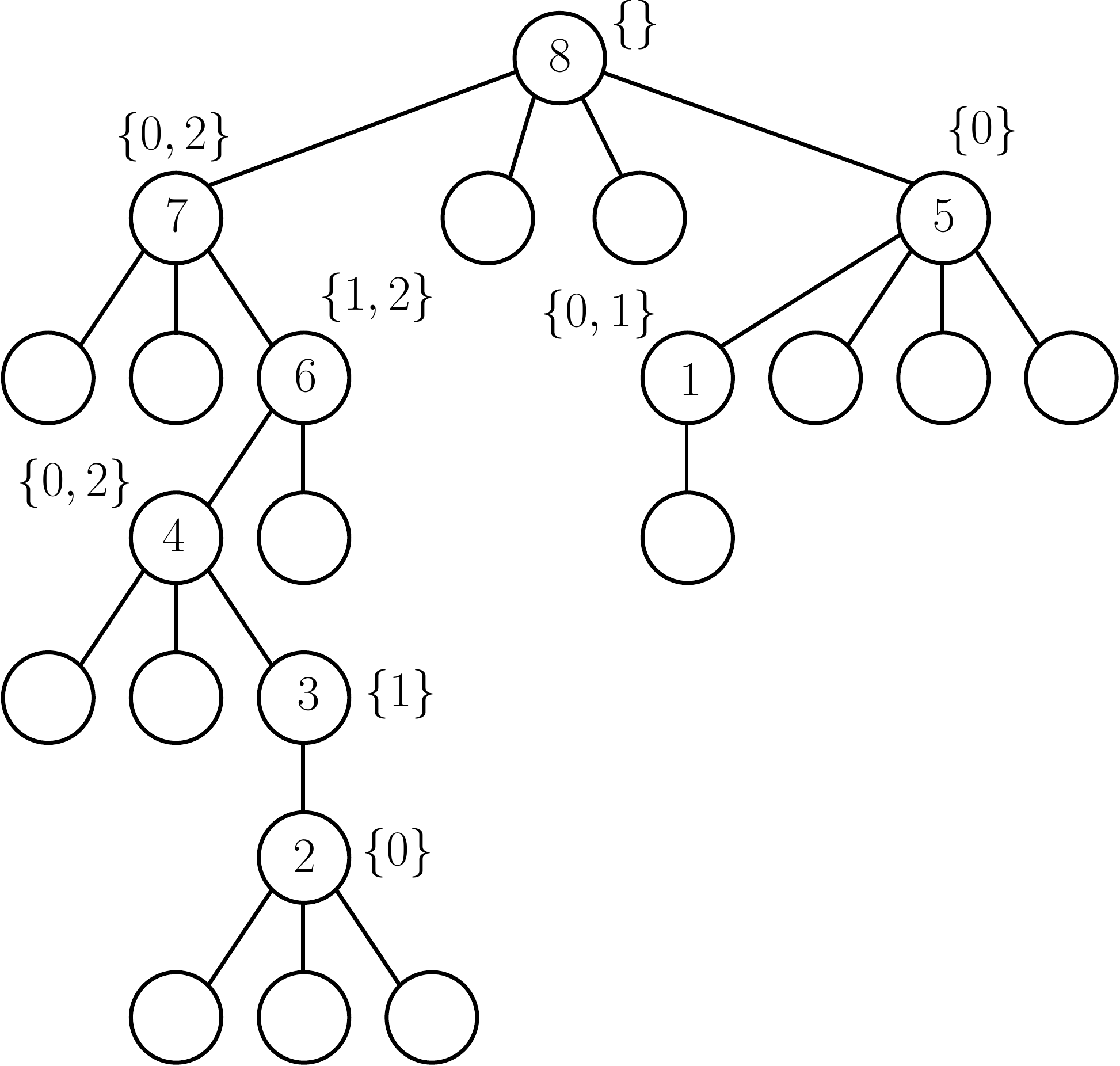}\label{fig:finalstep_no}}
    \caption{Two examples of the algorithm answering YES and NO}\label{fig:algoritmo}
  \end{centering}
\end{figure}

%\section*{References}

\end{document}